\documentclass[12pt, twoside]{article}
\usepackage{amsmath,amsthm,amssymb}
\usepackage{times}
\usepackage{enumerate}

\usepackage[nohug]{diagrams}

\def\titlerunning#1{\gdef\titrun{#1}}
\makeatletter
\def\author#1{\gdef\autrun{\def\and{\unskip, }#1}\gdef\@author{#1}}
\def\address#1{{\def\and{\\\hspace*{18pt}}\renewcommand{\thefootnote}{}%
\footnote {#1}}%
\markboth{\autrun}{\titrun}}
\makeatother
\def\email#1{e-mail: #1}
\def\subjclass#1{{\renewcommand{\thefootnote}{}%
\footnote{\emph{Mathematics Subject Classification (2010):} #1}}}
\def\keywords#1{\par\medskip
\noindent\textbf{Keywords.} #1}

\newtheorem{thm}{Theorem}[section]
\newtheorem{cor}[thm]{Corollary}
\newtheorem{lem}[thm]{Lemma}

\newtheorem{prop}[thm]{Proposition}

\theoremstyle{definition}
\newtheorem{defin}[thm]{Definition}
\newtheorem{rem}[thm]{Remark}

\numberwithin{equation}{section}

\newcommand{\set}[1]{\,\left\{#1\right\}}
\newcommand{\setd}[2]{\,\left\{#1\ \colon\ #2\right\}}

\newcommand{\sign}{\operatorname{sign}}

\newcommand{\RR}{\mathbb{R}}
\newcommand{\ZZ}{\mathbb{Z}}

\newcommand{\NN}{\mathbb{N}}

\begin{document}




\titlerunning{ }

\title{Poincar\'{e} inequalities and rigidity for actions on Banach spaces}

\author{Piotr W. Nowak}

\date{}

\maketitle

\address{Institute of Mathematics of the Polish Academy of Sciences, {\'{S}niadeckich} 8, 00-956 Warszawa, Poland
\and
Institute of Mathematics, University of Warsaw, Banacha 2, 02-097 Warszawa, Poland. \email{pnowak@mimuw.edu.p}}

\subjclass{Primary 22D12; Secondary 46B03}


\begin{abstract}
The aim of this paper is to extend the framework of the spectral method for proving property (T) to 
the class of reflexive Banach spaces and present
a condition implying that every affine isometric action of a given group
$G$ on a reflexive Banach space $X$ has a fixed point. This last property is a strong version of 
Kazhdan's property (T) and is equivalent to the fact that $H^1(G,\pi)=0$ for every
isometric representation $\pi$ of $G$ on $X$. The condition is expressed in
terms of $p$-Poincar\'{e} constants and we provide examples 
of groups, which satisfy such conditions and for which $H^1(G,\pi)$ vanishes for every
 isometric representation $\pi$ on an $L_p$ space for some $p>2$. 
 Our methods allow to estimate such a $p$ explicitly and yield several interesting applications.
In particular, we obtain quantitative estimates for 
vanishing of 1-cohomology with 
coefficients in uniformly bounded representations on a Hilbert space. We also 
give lower bounds on the conformal dimension of the boundary of
a hyperbolic group in the Gromov density model. 

\keywords{Poincar\'{e} inequality; Kazhdan's property (T); affine isometric action; 1-cohomology.}
\end{abstract}

\section{Introduction}

Kazhdan's property (T) is a powerful rigidity property of groups with numerous applications
and several characterizations. In this article we focus on the following description of property (T):
\emph{ a group $G$ has property (T) if and only if every affine isometric action of $G$ on the Hilbert
space has a fixed point}. This characterization can be  rephrased  as the cohomological
condition $H^1(G,\pi)=0$, for every unitary representation $\pi$ of $G$.
A generalization of property (T) to other Banach spaces
is then straightforward: we
are interested in  conditions implying that every affine isometric action of a given group on a given
Banach space has a fixed point. 
Such rigidity properties for actions on Banach spaces, as well as other generalizations of property (T),
and their applications, were studied earlier in
\cite{bader-et-al,fisher-margulis,chatterji-drutu-haglund,lafforgue}.

One very successful method of proving property (T)
is through spectral conditions on links of vertices of complexes acted upon by a group. 
Variations of such conditions were studied in 
\cite{ballman-swiatkowski,dymara-januszkiewicz,ershov-et-al,garland,
izeki-nayatani,kassabov,pansu,zuk-comptes,zuk-gafa,wang-jdg,wang-cag} in the context
of Hilbert spaces and non-positively curved spaces.
Given a group $G$ acting on a 2-dimensional simplicial complex,
one considers the link of a vertex. This link is a finite graph. If for every vertex, the
first positive eigenvalue of the discrete Laplacian is strictly larger than $1/2$, then 
$G$ has property (T).

The main purpose of this work is to extend the framework of the spectral method, and 
some of the rigidity results, beyond Hilbert spaces.  
Our main result provides such a framework
for the class of reflexive Banach spaces. The difficulty lies in the fact, 
that in the Hilbert space case the
spectral method  relies heavily on orthogonality, in particular self-duality
of Hilbert spaces. When passing to other Banach spaces, dual spaces
of certain Banach spaces and of their subspaces have to be identified,
and this is often a difficult task. We show, that when the representation is isometric such
computations are possible and we can use duality effectively. 

We focus on link graphs constructed using generating sets of a group, as
 in \cite{zuk-gafa}.  For a finite, symmetric generating set  $S$ not containing the identity element
the vertices of the link graph $\mathcal{L}(S)$ are the elements of $S$; generators $s$ and $t$ are
connected by an edge if $s^{-1}t$ is a generator. We will also assume that the graph is equipped with a 
weight $\omega$ on the edges.

Let $X$ be a Banach space and denote by $\kappa_p(S,X)$ be the optimal constant in 
the $p$-Poincar\'{e} inequality for the link graph $\mathcal{L}(S)$ of $G$ and the norm of $X$,
$$\sum_{s\in S} \Vert f(s)-Af\Vert_X^p\deg_{\omega}(s) \le \kappa_p^p \sum_{s\sim t} \Vert f(s)-f(t)\Vert_X^p\omega(s,t),$$
where $Af$ is the mean value of $f$.
When $X=L_2$, the constant $\kappa_2(S,L_2)=\kappa_2(S,\RR)$ can be expressed in terms of 
of the first eigenvalue of the discrete Laplacian. 

Our main result shows that 
sufficiently small constants  in Poincar\'{e} inequalities
for the graph $\mathcal{L}(S)$ imply the required cohomological vanishing. 
Given a number $1<p<\infty$ we denote by $p^*$ its adjoint index, satisfying $\frac{1}{p}+\frac{1}{p^*}=1$.

\begin{thm}\label{main theorem : intro}
Let $X$ be a reflexive Banach space and let 
$G$ be a group generated by a finite, symmetric set $S$, not containing the
identity element. If the link graph $\mathcal{L}(S)$ is connected and for some $1<p<\infty$
the associated Poincar\'{e} constants satisfy
$$\max\set{2^{-\frac{1}{p}} \kappa_{p}(S,X), 2^{-\frac{1}{p^*}} \kappa_{p^*}(S,X^*)}<1,$$
then 
$$H^1(G,\pi)=0,$$
for every isometric representation $\pi$ of $G$ on $X$.
\end{thm}

Clearly, by reflexivity, the same conclusion holds for actions on $X^*$.
Interestingly, the roles of the the two constants in the proof of the above theorem
are not symmetric.
 
We apply Theorem \ref{main theorem : intro} to $L_p$ spaces. The interesting case is $p>2$.
Indeed, when $1<p\le 2$, affine isometric actions exhibit the same behavior 
as for the Hilbert space: $G$ has property (T) if and
only if
any affine action on an $L_p$ space for $1<p\le 2$ has a fixed point \cite{bader-et-al}. 
Also, $G$ admits a metrically proper affine isometric action on the Hilbert space (i.e., is 
a-T-menable) if and only if it admits such an action on any $L_p[0,1]$ for $1<p\le 2$ \cite{nowak}
(see corrected version \cite{nowak-corrected}).
This last property is a strong negation of the existence of a fixed point.

Fixed point properties for groups acting on $L_p$ spaces for $p>2$ are difficult to prove
and only a handful of results are known:
\begin{enumerate}
\item higher rank algebraic groups and their lattices
have fixed points for every affine isometric action on  $L_p$-spaces for all $p> 1$  \cite{bader-et-al}; 
\item in  \cite{mimura} it was proved that $\operatorname{SL}_n(\ZZ[x_1,\dots,x_k])$ has fixed points
for every affine isometric on $L_p$ for every $p>1$ and $n\ge 4$;
\item Naor and Silberman  \cite{naor-silberman} showed that Gromov's random groups, 
containing (in a certain weak sense) expanders in their Cayley graphs, have a fixed point for affine isometric actions
on any $L_p$ for $p>1$;
\item  a general argument  due to Fisher and Margulis (see the proof in \cite{bader-et-al}) 
shows, that for every property (T)
group $G$ there exists a constant $\varepsilon=\varepsilon(G)>0$ such that any affine isometric
action on $L_p$ for $p\in [2,2+\varepsilon)$, has a fixed point. However, their argument does not
give any control over $\varepsilon$.
\end{enumerate}

On the other hand there are also groups which have property (T) but act without 
fixed points on $L_p$ spaces.
One example is furnished by $\operatorname{Sp}(n,1)$, which has property (T)
but has non-vanishing $L_p$-cohomology for $p>4n+2$, by a result of Pansu \cite{pansu-95}.
It also known that there exist hyperbolic groups which have property (T). Nevertheless, 
Bourdon and Pajot \cite{bourdon-pajot} showed that for every hyperbolic group $G$ 
and sufficiently large $p>2$ 
 there is an
affine isometric action on $\ell_p(G)$, whose linear part is the regular representation and which does 
not have a fixed point.
Moreover, Yu \cite{yu-hyperbolic} showed that every hyperbolic group admits
a proper, affine isometric action on $\ell_p(G\times G)$ for all sufficiently large $p>2$ (see also \cite{nica} for another construction).
We refer to \cite{nowak-survey} for a recent survey.

The techniques we use to establish the fixed point properties 
are different from the ones used previously
for general Banach spaces. In particular, we do not need the Howe-Moore 
property to prove our results. This representation theoretic property was 
necessary in \cite{bader-et-al,mimura}.
The expected outcome is also slightly different,
as our methods are not expected to give
fixed points on $L_p$ for all $p>1$.
One reason is that the $p$-Poincar\'{e} constants usually increase above $2^{1/p}$ 
as $p$ grows to infinity.
The second reason is that the main result applies to random hyperbolic groups, which, as 
remarked earlier, act without fixed points on $L_p$-spaces for $p>2$ sufficiently large.
Using our approach we obtain the appropriate vanishing of cohomology $H^1(G,\pi)$ for isometric representations $\pi$ on $L_p$-spaces with
$p\in [2,2+c)$, where the value of  $c$ depends on the group and can be estimated explicitly.
Finally, we point out that our techniques and the Poincar\'{e} inequalities we use, are all linear, in contrast to the non-linear approach used e.g., in \cite{izeki-nayatani,wang-jdg}.
Linearity allows to use interpolation methods effectively and also to obtain additional information about the structure of cohomology in the presence of 
spectral gaps.

To apply Theorem \ref{main theorem : intro} we need to estimate
$p$-Poincar\'{e} constants for $p>2$. Even in classical settings, such as
convex domains in $\RR^n$, estimates exist but exact values of $p$-Poincar\'{e} constants are not known,
except a few special cases.
The situation is even worse for finite graphs, where very few estimates are known 
for cases other than $p=1,2$.
Here we consider the family  of $\widetilde{A}_2$-groups, indexed by powers of primes.
These groups were introduced and studied in \cite{cartwright-et-al-GD,cartwright-et-al}. For every $q$, the
group $G_q$ has a generating 
sets whose link graph is the incidence graph of the finite projective plane over the field $\mathbb{F}_q$.
Spectra of
such graphs were computed in \cite{feit-higman} and 
give, in particular, the exact value of the Poincar\'{e} constant $\kappa_2(S,\mathbb{R})$. 
We use this fact to estimate $\kappa_p(S,L_p)$ 
for these graphs, which allows to obtain for each $q$ a number $c_q$ such that any
 affine isometric
action of $G$ on any $L_p$ has a fixed point for $p\in [2,2+c_q)$. The explicit estimates
of $c_q$ are given in Theorem \ref{theorem : fixed points for A_2-groups}.

As mentioned earlier, our results apply to random hyperbolic groups,
more precisely, to random groups
in the Gromov density model with densities $1/3<d<1/2$, and yield important consequences. These
groups are hyperbolic and have Kazhdan's property (T) with overwhelming probability
\cite{zuk-gafa,kotowscy}.
Using our methods we give lower bounds on $p$ for which fixed points exists for all isometric actions on any 
$L_p$-space. A connection with the conformal dimension arises through  the work 
of Bourdon and Pajot \cite{bourdon-pajot} and allows us to give a lower bound on the 
conformal
dimension of a boundary of a  random hyperbolic group, using an associated link graph, see Section \ref{section : hyperbolic groups}. 
The problem of estimating the conformal dimension of random hyperbolic groups was posed by Gromov \cite[9.B (g)]{gromov} and Pansu \cite[IV.b]{pansu-conformal}.

Our methods also apply to affine actions, whose linear
part is a  uniformly bounded representation on a Hilbert space.
More precisely, we show that $H^1(G,\pi)=0$, whenever $\pi$ is a uniformly bounded 
representation with norms of all operators bounded by a  constant, which depends on
 the group but is close to $\sqrt{2}$ in many cases, see Theorems \ref{theorem: A2tilde vanishing and Banach Mazur distance} and \ref{theorem: vanishing hyperbolic and Banach Mazur distanc}. The question of extending property (T) in the form of cohomological vanishing from unitary to uniformly bounded representation is a
 well-known open question. In particular, Shalom conjectures that for every hyperbolic group there exists a uniformly bounded representation
 with a proper cocycle. The case of $\operatorname{Sp}(n,1)$ is an unpublished result of Shalom.

Finally, we present other applications. We improve the differentiability class of diffeomorphic
 actions
on the circle in the rigidity theorem in \cite{navas-fr,navas} and  estimate eigenvalues of the 
discrete $p$-Laplacian
on finite quotients of groups using Kazhdan-type constants.

\tableofcontents

\section{Actions on Banach spaces}\label{section : actions on Banach spaces}

\subsection{Generating sets and link graphs} 
Let $G$ denote a discrete group generated by a finite symmetric set $S=S^{-1}$.
Let $\mathcal{L}(S)$ denote the following graph, called the link graph of $S$. The vertices are given by 
$\mathcal{V}=S$.
Two vertices  $s,t\in S$ are connected by an edge, denoted $s\sim t$,
if and only if $s^{-1}t\in S$ and $t^{-1}s\in S$.

The set $E$ is defined as follows:
$$E=\setd{(s,t)\in S\times S}{s^{-1}t\in S}.$$
Note that $E$ can be viewed as the set of oriented edges and in $E$ every edge is counted twice.

A \emph{weight} on $\mathcal{L}(S)$ is a function $\omega:E\to (0,\infty)$, such that 
$\omega(s,t)=\omega(t,s)$,
for every $s,t\in S$.
Given a weight on the link graph, the associated \emph{degree} of a vertex $r\in S$ is defined to be
$$\deg_{\omega}(s)= \sum_{t,\ t\sim s} \omega(t,s).$$
A weight $\omega$ on a link graph $\mathcal{L}(S)$ is \emph{admissible} if it satisfies
\begin{enumerate}
\item $\deg_{\omega}(s)=\deg_{\omega}(s^{-1})$, and
\item $\deg_{\omega}(r)=\sum_{(s,t):\ s^{-1}t=r} \omega(s,t)$,
\end{enumerate}
for every $r,s,t\in S$. Note that
$$\sum_{s\in S}\deg_{\omega}(s)=\omega(E).$$
Throughout the article we consider only admissible weights on link graphs of generating sets.

\subsection{Isometric representations and associated Banach spaces}
Let $X$ be a Banach space equipped with a norm $\Vert \cdot \Vert_X$. 
We assume throughout that $X$ is reflexive and that 
$\pi:G\to B(X)$ is 
a representation of $G$ into the bounded invertible operators on $X$. 
Let $X^*$ denote the continuous dual of $X$, with its standard norm.
$X^*$ is naturally equipped with the adjoint representation of $G$, $\overline{\pi}:G\to B(X^*)$, 
$$\overline{\pi}_g =\pi_{g^{-1}}^*.$$
Throughout we fix $1<p<\infty$. The value of $p$ will be chosen later 
depending on the context. We denote by $p^*$  the adjoint index, satisfying $\frac{1}{p}+\frac{1}{p^*}=1$,
and by $L_p$ the space $L_p(\mu)$ for any measure $\mu$ (our results apply with no assumptions
on the measure).
We also use $\simeq$ to denote an isomorphism and $\cong$ to denote an isometric isomorphism
of Banach spaces. 

Define the Banach space $C^{(0,p)} (G,\pi)$ to be the linear space $X$, with the norm
$$\Vert v \Vert_{(0,p)}=\omega(E)^{\frac{1}{p}} \Vert v\Vert_X.$$
Let $\langle\cdot\,,\cdot\rangle_X$ denote the natural pairing between $X$ and $X^*$.
The pairing between $C^{(0,p)} (G,\pi)$ and $C^{(0,p^*)} (G,\overline{\pi})$ is given by
$$\langle v,w\rangle_0=\omega(E)\,\langle v,w\rangle_X.$$
Then $C^{(0,p^*)} (G,\overline{\pi})$ is the dual space of $C^{(0,p)} (G,\pi)$.

We define $C^{(1,p)} (G,\pi)$ to be the finite direct sum $\bigoplus_{s\in S}X$, with the norm given by
$$\Vert f\Vert_{(1,p)}=\left(\sum_{s\in S} \Vert f(s)\Vert_X^p\deg_{\omega}(s)\right)^{\frac{1}{p}}.$$
The dual of $C^{(1,p)} (G,\pi)$ is $C^{(1,p^*)} (G,\overline{\pi})$, via the pairing
$$\langle  f,\phi\rangle_1=\sum_{s\in S}\langle f(s),\phi(s)\rangle_X\deg_{\omega}(s),$$
 for $f\in C^{(1,p)} (G,\pi)$ and $\phi\in C^{(1,p)} (G,\pi)^*$.
 
Define an  operator $Q_{\pi}$ on $C^{(1,p)}(G,\pi)$, 
$$Q_{\pi}f(s)=\pi_s f(s^{-1}).$$
A similar operator $Q_{\overline{\pi}}$ is defined on $C^{(1,p^*)}(G,\overline{\pi})$. The following is straightforward to verify.
\begin{lem}
The operator $Q_{\pi}$ is an involution, satisfying $Q_{\pi}^*=Q_{\overline{\pi}}$.
\end{lem}

Consider the following subspaces of $C^{(1,p)}(G,\pi)$, defined as eigenspaces of $Q_{\pi}$:
$$C^{(1,p)}_+(G,\pi)=\setd{f\in C^{(1,p)}(G,\pi)}{f=Q_{\pi} f},$$
and
$$C^{(1,p)}_- (G,\pi)=\setd{f\in C^{(1,p)} (G,\pi)}{f=-Q_{\pi} f}.$$

\begin{lem}
For any $1<p<\infty$ we have $C^{(1,p)} (G,\pi)=C^{(1,p)}_+ (G,\pi)\oplus C^{(1,p)}_- (G,\pi)$.
\end{lem}

\begin{proof}
We define two bounded operators: $P_{\pi}^+:C^{(1,p)} (G,\pi)\to C^{(1,p)}_+ (G,\pi)$,
$$P_{\pi}^+=\dfrac{I+Q_{\pi}}{2},$$
and $P_{\pi}^-:C^{(1,p)} (G,\pi)\to C^{(1,p)}_- (G,\pi)$,
$$P_{\pi}^-=\dfrac{I-Q_{\pi}}{2}.$$
Clearly $P_{\pi}^++P_{\pi}^-=I$.
Additionally, $C^{(1,p)}_+ (G,\pi)=\ker P_{\pi}^-=\operatorname{im}P_{\pi}^+$ and 
$C^{(1,p)}_- (G,\pi)=\ker P_{\pi}^+=\operatorname{im}P_{\pi}^-$. Indeed,
we have 
$$\pi_{s^{-1}}(P_{\pi}^+f(s))=\dfrac{\pi_{s^{-1}}f(s)+f(s^{-1})}{2}=P_{\pi}^+f(s^{-1}).$$
Finally, $P_{\pi}^+$ restricted to $C^{(1,p)}_+ (G,\pi)$ and $P_{\pi}^-$ restricted to $C^{(1,p)}_- (G,\pi)$ are identity operators, 
so that $P_{\pi}^+$ and $P_{\pi}^-$ are projections onto the required subspaces.
\end{proof}

We now  analyze the structure of $C^{(1,p)} (G,\pi)$  in relation to the one of $C^{(1,p^*)} (G,\overline{\pi})$.\\


\subsection{Duality for $C^{(1,p)}_- (G,\pi)$}
The dual of $C^{(1,p)} (G,\pi)$ is $C^{(1,p^*)} (G,\overline{\pi})$.
Let $P_{\overline{\pi}}^+:C^{(1,p^*)} (G,\overline{\pi})\to C^{(1,p^*)}_+ (G,\overline{\pi})$ and  
$P_{\overline{\pi}}^-:C^{(1,p^*)} (G,\overline{\pi})\to C^{(1,p^*)}_- (G,\overline{\pi})$ denote similar projections as above on the dual level.

\begin{lem}
We have $P_{\overline{\pi}}^+=(P_{\pi}^+)^*$ and $P_{\overline{\pi}}^-=(P_{\pi}^-)^*$.
\end{lem}
\begin{proof}
Let $f\in C^{(1,p)} (G,\pi)$ and $\phi\in C^{(1,p^*)} (G,\overline{\pi})$. Then
$$(P_{\pi}^-)^*=\dfrac{1}{2}\left(I-Q_{\pi}\right)^*=\dfrac{1}{2}\left(I-Q_{\overline{\pi}}\right)=P_{\overline{\pi}}^-.$$
Similarly for $P_{\pi}^+$.
\end{proof}

\begin{lem}
We have the following isomorphisms: $C^{(1,p)}_- (G,\pi)^*\simeq C^{(1,p^*)}_- (G,\overline{\pi})$ and 
$C^{(1,p)}_+ (G,\pi)^*\simeq C^{(1,p^*)}_+ (G,\overline{\pi})$.
\end{lem}
\begin{proof}
Consider $f\in C^{(1,p)}_- (G,\pi)$ and let $\phi\in C^{(1,p^*)} (G,\overline{\pi})$. Then
$$\langle f,\phi\rangle_1=\langle -Q_{\pi}f,\phi\rangle_1=\langle  f, -Q_{\overline{\pi}} \phi\rangle_1.$$
Therefore,
$$2\langle f, P_{\overline{\pi}}^+\phi\rangle_1=0,$$
which shows that $C^{(1,p^*)}_+ (G,\overline{\pi})$ annihilates $C^{(1,p)}_- (G,\pi)$. 

Conversely, if $\phi\in C^{(1,p^*)} (G,\overline{\pi})$ annihilates $C^{(1,p)}_- (G,\pi)$, then 
$$\langle P_{\pi}^-f,\phi\rangle_1 =\langle f, P_{\overline{\pi}}^-\phi\rangle=0$$
for every $f\in C^{(1,p)} (G,\pi)$. Consequently,  $P_{\overline{\pi}}^-\phi=0$ and $\phi=P_{\overline{\pi}}^+\phi$, which
means it belongs to $C^{(1,p^*)}_+ (G,\overline{\pi})$. Thus,
$$C^{(1,p)}_- (G,\pi)^*\cong C^{(1,p)} (G,\overline{\pi})\Big/C^{(1,p^*)}_+ (G,\overline{\pi})\simeq C^{(1,p^*)}_- (G,\overline{\pi}).$$
Other cases are proved similarly.
\end{proof}
However, in order to identify the dual of $C^{(1,p)}_- (G,\pi)$ an isomorphism is not sufficient,
we need an isometric isomorphism instead. For a representation $\pi$, $C^{(1,p)}_- (G,\pi)^*$ is 
in general not isometrically isomorphic to $C^{(1,p^*)}_- (G,\overline{\pi})$.
However, it turns out that this additional property holds when the representation $\pi$ is isometric.

\begin{thm}\label{theorem : duality for C_-}
Assume that $\pi_s$ is an isometry for every $s\in S$. Then we have the following isometric isomorphisms:
$C^{(1,p)}_- (G,\pi)^*\cong C^{(1,p^*)}_- (G,\overline{\pi})$ and $C^{(1,p)}_+ (G,\pi)^*\cong C^{(1,p^*)}_+ (G,\overline{\pi})$. 
\end{thm}

\begin{proof}
Consider $C^{(1,p^*)} (G,\overline{\pi}) \Big/C^{(1,p^*)}_+ (G,\overline{\pi})$, which consists 
of cosets $[\phi]=C^{(1,p^*)}_+ (G,\overline{\pi})+\phi$, for $\phi\in C^{(1,p^*)} (G,\overline{\pi})$. 
We need to show that  for each such coset $N$,
$\inf\setd{ \Vert \phi\Vert}{N=[\phi]}$ is attained when $\phi\in C^{(1,p^*)}_- (G,\overline{\pi})$.

For $\phi\in C^{(1,p^*)}_- (G,\overline{\pi})$ and $\psi\in C^{(1,p^*)}_+ (G,\overline{\pi})$, we have
$$ \Vert \phi+\psi\Vert_{(1,p^*)}=\Vert -Q_{\overline{\pi}}\phi + Q_{\overline{\pi}} \psi\Vert_{(1,p^*)}=\Vert \phi-\psi\Vert_{(1,p^*)},$$
since the involution $Q_{\overline{\pi}}$ is an isometry, whenever $\pi$, or equivalently $\overline{\pi}$, is an isometric representation.
Now consider the coset $[\phi]$ for $\phi\in C^{(1,p^*)}_- (G,\overline{\pi})$ and 
consider another element, $\zeta\in C^{(1,p^*)} (G,\overline{\pi})$, such that 
$\zeta-\phi\in C^{(1,p^*)}_+ (G,\overline{\pi})$, so that $\zeta=\phi+\psi$, for some  $\psi\in C^{(1,p^*)}_+ (G,\overline{\pi})$. 
This implies 
$$\Vert \phi\Vert_{(1,p^*)}\le \dfrac{\Vert \phi-\psi\Vert_{(1,p^*)}+
\Vert \phi+\psi\Vert_{(1,p^*)}}{2}=\Vert \zeta\Vert_{(1,p^*)},$$
which proves the claim.
\end{proof}
This last statement allows us to identify $C^{(1,p)}_- (G,\pi)^*$ with $C^{(1,p^*)}_- (G,\overline{\pi})$ for isometric 
representations
and is crucial in the proof of the main theorem.

\subsection{The operator $\delta$}
We define the operator  $\delta_{\pi}:C^{(0,p)} (G,\pi)\to C^{(1,p)}_- (G,\pi)$ by the formula
$$\delta_{\pi} v(s)=v-\pi_sv.$$
Theorem \ref{theorem : duality for C_-} allows to express the adjoint of $\delta_{\pi}$ in a 
way, which is convenient for calculations. We have the following explicit formula for  $\delta_{\pi}^*$.
\begin{lem}
The operator
$\delta_{\pi}^*:C^{(1,p^*)}_- (G,\overline{\pi})\to C^{(0,p^*)} (G,\overline{\pi})$ is given by 
\begin{equation}\label{equation : formula for delta*}
\delta_{\pi}^*\phi= 2\sum_{s\in S}\phi(s)\dfrac{\deg_{\omega}(s)}{\omega(E)}.
\end{equation}
\end{lem}
\begin{proof}
\begin{eqnarray*}
\langle \delta_{\pi} v, \phi\rangle_1&=&\sum_{s\in S}\langle v-\pi_sv,\phi(s)\rangle_X\deg_{\omega}(s)\\
&=&\sum_{s\in S}\left(\langle v,\phi(s)\rangle_X-\langle v,\overline{\pi}_{s^{-1}}\phi(s)\rangle_X\right)\deg_{\omega}(s)\\
&=&\sum_{s\in S}\left(\langle v,\phi(s),\rangle_X+\langle v,\phi(s^{-1})\rangle_X\right)\deg_{\omega}(s)\\
&=&\left\langle v, 2\sum_{s\in S}\phi(s)\dfrac{\deg_{\omega}(s)}{\omega(E)}\right\rangle_0. 
\end{eqnarray*}
\end{proof}

It is now clear that $\delta_{\pi}^*$ admits a continuous extension to the space $C^{(1,p^*)} (G,\overline{\pi})$, 
defined by the right hand side of the formula (\ref{equation : formula for delta*}).

\subsection{The operators $D$, $L$, and $d$}
We define the Banach space,
$$C^{(2,p)} (G,\pi)=\setd{ \eta\in \bigoplus_{(s,t)\in E}X}{\eta(s,t)=-\eta(t,s)},$$
equipped with the norm 
$$\Vert \eta\Vert_{(2,p)}=\left(\sum_{(s,t)\in E} \Vert \eta(s,t)\Vert^p_X\omega(s,t)\right)^{\frac{1}{p}}.$$
We also define operators $D,L_{\pi}:C^{(1,p)}_- (G,\pi)\to C^{(2,p)} (G,\pi)$ by the formulas
$$D f(s,t)=f(t)-f(s),$$ 
$$L_{\pi} f(s,t)=\pi_sf(s^{-1}t).$$
Then the operator $d_{\pi}$ is defined by 
$$d_{\pi}=L_{\pi}-D.$$
Similarly, define $\overline{D}$, $L_{\overline{\pi}}$ and $d_{\overline{\pi}}$ for the adjoint  representation.

\begin{lem}
Let $\pi$ be an isometric representation.
The operator $L_{\pi}$ is an isometry onto its image. Consequently, 
$D$ is an isometry as well when restricted to $\ker d_{\pi}$. (The same claim holds for
$L_{\overline{\pi}}$ and $\overline{D}$, restricted to $\ker d_{\overline{\pi}}$).

\end{lem}
\begin{proof}
By direct calculation,
\begin{eqnarray*}
\Vert L_{\pi}f\Vert_{(2,p)}^p&=&\sum_{(s,t)\in E} \Vert \pi_sf(s^{-1}t)\Vert_X^p\omega(s,t)\\
&=&\sum_{s\in S}\Vert f(s)\Vert_X^p\deg_{\omega}(s)\\
&=&\Vert f\Vert_{(1,p)}^p.
\end{eqnarray*}
\end{proof}

The kernel of the operator $\overline{D}$ consists of the constant functions on $S$, 
which is a complemented 
subspace of $C^{(1,p)} (G,\pi)$. The projection onto this subspace is given by
$$\overline{M}\phi(s)=\sum_{s\in S}\phi(s)\dfrac{\deg_{\omega}(s)}{\omega(E)}.$$
Note that for $\phi\in C^{(1,p^*)}_- (G,\overline{\pi})$ we have
$$\overline{M}\phi(s)=\dfrac{1}{2}\delta_{\pi}^*\phi,$$
for every $s\in S$.

\begin{lem}\label{lemma : norm of Mf and delta f are equal}
Let $\phi\in C^{(1,p^*)}_- (G,\overline{\pi})$. Then  
$\Vert \overline{M}\phi\Vert_{(1,p^*)}=\dfrac{1}{2}\Vert \delta_{\pi}^*\phi \Vert_{(0,p^*)}$.
\end{lem}
\begin{proof}
We have the following equalities:
\begin{eqnarray*}
\Vert \overline{M}\phi\Vert_{(1,p^*)}^{p^*}&=&\sum_{s\in S}\left\Vert \dfrac{\delta_{\pi}^* \phi}{2}\right\Vert_X^{p^*}
\deg_{\omega}(s)\\
&=&\dfrac{1}{2^{p^*}}\Vert \delta_{\pi}^* \phi\Vert^{p^*}_X \left(\sum_{s\in S}\deg_{\omega}(s) \right)\\
&=&\dfrac{1}{2^{p^*} }\Vert \delta_{\pi}^* \phi\Vert^{p^*}_0.
\end{eqnarray*}
\end{proof}

\subsection{Sufficients conditions for vanishing of cohomology} 
Given a group $G$, the 1-cocycles associated to $\pi$ are functions $b:G\to X$ satisfying
the cocycle condition,
$$b_g=\pi_gb_h+b_g,$$
for every $g,h\in G$.
The coboundaries are those cocycles which are of the form $$b_g=v-\pi_gv$$ for some $v\in X$ and 
all $g\in G$. The first cohomology of $G$ with coefficients in $\pi$ is defined to be
$H^1(G,\pi)=\text{cocycles}\big/\text{coboundaries}$.

An affine action of $G$ on $X$ is defined as
$$A_g v=\pi_gv+b_g,$$
where $\pi$ is called the linear part of the action and $b$ is a cocycle.
Vanishing of cohomology $H^1(G,\pi)$ is equivalent to the existence of a fixed point
for any affine action with linear part $\pi$.
We refer to \cite{bekka-et-al} for background on cohomology and affine actions.

The reader can easily verify the following lemma.
\begin{lem}
$\operatorname{image}(\delta_{\pi})\subseteq \ker d_{\pi}$. 
\end{lem}

This fact allows to formulate the following sufficient condition for the fixed point property for affine actions on $X$.

\begin{prop}\label{proposition : delta onto ker d implies fixed point}
If  the image of $\delta_{\pi}$ is equal to $\ker d_{\pi}$,
then $H^1(G,\pi)=0$.
\end{prop}
\begin{proof}
Let $b:G\to X$ be a 1-cocycle for $\pi$ and let $b'$ denote the restriction of $b$
to the generating set $S$. The cocycle condition implies that $b'\in C^{(1,p)}_- (G,\pi)$ and, 
furthermore,  that $b'\in \ker d_{\pi}$. 
If $\delta_{\pi}$ is onto $\ker d_{\pi}$, then $b'=\delta_{\pi} v$ for some $v\in X$. Since $b$ is trivial on the
generators, we conclude  that $b$ is trivial.
\end{proof}
It is important to remark that the technical details here are slightly different than
in \cite{zuk-gafa}, where the original condition in terms of almost invariant vectors is deduced,
and one needs to use the Delorme-Guichardet theorem to obtain cohomological vanishing.
The above argument allows to bypass the use of the Delorme-Guichardet theorem and obtain
vanishing of cohomology directly.

Note that the image of $\delta_{\pi}$ is always properly contained in $C^{(1,p)}_- (G,\pi)$.
By the open mapping theorem we also have the following 
\begin{cor}\label{corollary : Poincare constant on Cayley graph}
Assume $\pi$ does not have invariant vectors. If $\delta$ is onto $\ker d$
then there is a constant $K>0$ such that 
$$\sup_{s\in S}\Vert v-\pi_sv\Vert_X\ge K\Vert v\Vert_X,$$
for every $v\in X$.
\end{cor}
The constant $K$ in the above statement can be viewed as a version of Kazhdan constant 
for  isometric representations of $G$ on $X$.

\section{Poincar\'{e} inequalities associated to norms}

Consider a weighted, finite graph $\Gamma=(\mathcal{V},\mathcal{E})$, a number $p\ge 1$
and a Banach space $X$. The $p$-Poincar\'{e} inequality for $\Gamma$ and the norm of $X$
is the inequality
\begin{equation}\label{equation : standard p-X Poincare inequality on a graph}
\left(\sum_{x\in \mathcal{V}} \Vert f(x)-Af\Vert_X^p\deg_{\omega}(x)\right)^{\frac{1}{p}}\le \kappa_p \left(\sum_{x\sim y} \Vert f(x)-f(y)\Vert_X^p\omega(x,y)\right)^{\frac{1}{p}},
\end{equation}
for all functions $f:\mathcal{V}\to X$, where $Af=\frac{1}{2 \omega(\mathcal{E}) }\sum_{x\in \mathcal{V}} f(x)\deg_{\omega}(x)$.
On a finite graph, the inequality (\ref{equation : standard p-X Poincare inequality on a graph}) is always satisfied for some $\kappa_p>0$.

\begin{defin}
Let $\mathcal{L}(S)$ be a link graph of a generating set $S$, with weight $\omega$. 
For a Banach space $X$ and a number $1<p<\infty$ 
we define the constant  $\kappa_p(S,X)$
of $\mathcal{L}(S)$ by setting 
$$\kappa_p(S,X)=\inf \kappa_p,$$
where the infimum is taken over all $\kappa_p$, for which inequality 
(\ref{equation : standard p-X Poincare inequality on a graph}) holds.
\end{defin}
We will omit the reference to $\omega$ in the notation for $\kappa$.

\subsubsection{Hilbert spaces}
When $X=L_2$ is the Hilbert space this constant is related to the smallest positive eigenvalue $\lambda_1$ 
of the Laplacian 
on the graph  as follows:
$$\kappa_2(S,L_2)=\sqrt{\lambda_1^{-1}},$$
since the latter can be defined via the variational expression and the Rayleigh quotient.

\subsubsection{$L_p$-spaces, $1\le p<\infty$}
Let $(Y,\mu)$ be any measure space and for $X=\mathbb{R}$ consider a $p$-Poincar\'{e} inequality 
\begin{equation}\label{equation : discrete Poincare inequality}
\sum_{x\in \mathcal{V}}\vert f(x)-Af\vert^p\deg_{\omega}(s)\le \kappa_p^p\sum_{x\sim y}\vert f(x)-f(y)\vert^p\omega(x,y)
\end{equation}
on a finite graph. By integrating over $Y$ with respect to $\mu$ we obtain
$$\sum_{x\in \mathcal{V}}\Vert f(x)-Af\Vert_{L_p}^p\deg_{\omega}(x)\le \kappa_p^p\sum_{x\sim y}\Vert f(x)-f(y)\Vert_{L_p}^p\omega(x,y),$$
for any $f:\mathcal{V}\to L_p$.
This gives 
$$\Vert f-Af\Vert_{(1,p)}\le \kappa_p\Vert \nabla f\Vert_{(2,p)},$$
so that $\kappa_p(S,{L_p})$ is equal to $\kappa_p(S,\mathbb{R})$ in the inequality 
(\ref{equation : discrete Poincare inequality}).

\subsubsection{Direct sums}
More generally, consider an $\ell_p$-direct sum $X=\left(\bigoplus_{s\in S} X_i\right)_p$ of Banach 
spaces $\{X_i\}_{i\in I}$.
A similar argument as above shows that $\kappa_p(S,X)\le \sup_{i\in I}\kappa_p(S,X_i)$.

\subsubsection{The case $p=\infty$}
Consider $s, s^{-1}\in S$ and choose $x\in X$ such that 
$\Vert x\Vert_X=1$. Let $d_S$ denote the path metric on $\mathcal{L}(S)$. 
Define $f:\Gamma\to\RR$ by
the formula
$$f(t)=\left\lbrace\begin{array}{ll}
\left(1-2\dfrac{d_S(s,t)}{d_S(s,s^{-1})}\right) x&\text{ if } d_S(s,t)\le \dfrac{d_S(s,s^{-1})}{2},\\
\left(-1+2\dfrac{d_S(s,t)}{d_S(s,s^{-1})}\right) x&\text{ if } d_S(s^{-1},t)\le \dfrac{d_S(s,s^{-1})}{2},\\ 
0&\text{ if } d_S(s,t)> \dfrac{d_S(s,s^{-1})}{2}   \text{ and }   d_S(s^{-1},t)> \dfrac{d_S(s,s^{-1})}{2}.
\end{array}\right.
$$
For such $f$ we have  $\Vert f\Vert_{(X,\infty)}=1$ and 
$Af=0$, however $\Vert D f\Vert_{(X,\infty)}=\dfrac{1}{d_S(s,s^{-1})}$. Thus we have 
$$\kappa_{\infty} (G,\pi)\ge{\max_{s\in S} d_S(s,s^{-1})}$$
and for sufficiently large $S$, the above Poincar\'{e} constant is at least $1$. Additionally, for 
any $\varepsilon>0$ there exists a sufficiently large
$p<\infty$, such that the norms $\Vert f\Vert_{(1,p)}$ and $\Vert f\Vert_{(1,\infty)}$ 
are $\varepsilon$-close. For a sufficiently small $\varepsilon>0$ and the corresponding $p$
as above, we also have $2^{-\frac{1}{p}}\kappa_p(S,X)\ge 1$.

\subsubsection{Behavior under isomorphisms}\label{subsection : behavior under isomorphisms}
Let $T:X\to Y$ be an isomorphism of Banach spaces $X$, $Y$, satisfying 
$\Vert x\Vert_X\le \Vert Tx\Vert_y \le L \Vert x\Vert_X$ for every $x\in X$.
Then  $\kappa_p (G,\pi)\le L\kappa_p(G,Y)$.
\section{Vanishing of cohomology}

\subsection{An inequality for $\kappa_p$ and $\delta_{\pi}^*$}
Note that since in $E$ 
each edge of $\mathcal{L}(S)$ is counted twice, we have 
$\Vert Df \Vert_{(2,p)}=2^{\frac{1}{p}}\Vert \nabla f\Vert_{\ell_p(S,X)}$
and $Mf=Af$.
The following result describes the relation between Poincar\'{e} constants
and the operator $\delta_{\pi}^*$.

\begin{thm}\label{theorem : inequality for kappas}
The inequality 
\begin{equation}\label{equation : Kazhdan constant}
2 \left(1-2^{-\frac{1}{p^*}}\kappa_{p^*}(S,X)\right)\Vert \phi \Vert_{(1,p^*)}\le \Vert \delta_{\pi}^* \phi \Vert_{(0,p^*)},
\end{equation}
holds
for every $\phi\in \ker d_{\overline{\pi}}$.
\end{thm}
\begin{proof}
Let $\phi:S\to X^*$. Then
\begin{eqnarray*}
\kappa_{p^*}(S,X^*)\,\Vert \overline{D} \phi\Vert_{(2,p^*)}&=&\kappa_{p^*}(S,X^*)\, 2^{\frac{1}{p^*}}
 \Vert \overline{\nabla}\phi \Vert_{\ell_{p^*}(\mathcal{E},X)}\\
&\ge&2^{\frac{1}{p^*}}\left\Vert \phi-\overline{M}\phi\right\Vert_{(1,p^*)}\\
&\ge&2^{\frac{1}{p^*}}\left(\Vert \phi\Vert_{(1,p^*)}-\left\Vert \overline{M}\phi\right\Vert_{(1,p^*)}\right)\\
\end{eqnarray*}
Since $\overline{D}$ is an isometry on $\ker d_{\overline{\pi}}$, 
$$2^{\frac{1}{p^*}} \Vert \phi\Vert_{(1,p^*)}- \kappa_{p^*}(S,X^*)\Vert \phi\Vert_{(1,p^*)}\le 2^{\frac{1}{p^*}}\Vert 
\overline{M}\phi\Vert_{(1,p^*)},$$
which, by lemma \ref{lemma : norm of Mf and delta f are equal}, becomes
$$ \left(1-2^{-\frac{1}{p^*}} \kappa_{p^*}(S,X^*)\right)\Vert\phi\Vert_{(1,p^*)}\le  \dfrac{1}{2}
\Vert \delta_{\pi}^*\phi \Vert_{(0,p^*)}.$$
\end{proof}

\begin{rem}\normalfont
The above inequality does not reduce to the one in \cite{zuk-gafa} in the case $X=L_2$ and $p=2$,
even 
though in both cases the constant is non-zero if $\kappa_2(S,\mathbb{R})<\sqrt{2}$.
For $X=L_2$ and $p=2$, Theorem \ref{theorem : inequality for kappas} 
gives a strictly smaller lower estimate for the norm of
the operator $\delta_{\pi}^*$. Indeed, in that case the estimate obtained using spectral methods is 
$$\sqrt{2\left(2-\kappa_2(S,L_2)^2\right)}\ \Vert \phi\Vert_{(1,2)}\le \Vert \delta_{\pi}^*\phi\Vert_{(2,2)}.$$
This difference is a consequence of the fact that in the case $p=2$ and $X=L_2$ we can 
apply the Pythagorean theorem instead of the triangle inequality in the first sequence of inequalities. 

\end{rem}

A similar inequality as in Theorem \ref{theorem : inequality for kappas}
holds for $\kappa_p(S,X)$ and the norm of $\delta_{\overline{\pi}}^*$. 
The above inequality can be now used to show that
sufficiently small constants in Poincar\'{e} inequalities on the link graph
imply fixed point properties.

\subsection{Proof of the main theorem}

\newcounter{aux1}
\newcounter{aux2}

\setcounter{aux1}{\value{section}}
\setcounter{section}{1}

\setcounter{aux2}{\value{thm}}
\setcounter{thm}{0}

\begin{thm}
Let $X$ be a reflexive Banach space and let 
$G$ be a group generated by a finite, symmetric set $S$, not containing the
identity element. If the link graph $\mathcal{L}(S)$ is connected and for some $1<p<\infty$
the associated Poincar\'{e} constants satisfy
$$\max\set{2^{-\frac{1}{p}} \kappa_{p}(S,X), 2^{-\frac{1}{p^*}} \kappa_{p^*}(S,X^*)}<1,$$
then 
$$H^1(G,\pi)=0,$$
for every isometric representation $\pi$ of $G$ on $X$.
\end{thm}

\setcounter{section}{\value{aux1}}

\setcounter{thm}{\value{aux2}}

\begin{proof}
We have the following dual diagrams:
\begin{diagram}
&&\ker d_{\overline{\pi}}&&\\
&\ruTo(2,4)_{\ \delta_{\overline{\pi}}} &\dInto^{\overline{i}}&&\\
&&C^{(1,p^*)}_- (G,\overline{\pi}) & \pile{\rTo^{\ \ \ \ \ \ \ d_{\overline{\pi}}\ \ \ \ \ \ \ }\\ \rTo_{\overline{D}}}&C^{(2,p^*)} (G,\overline{\pi})\\ 
&\ldTo(2,2)_{\delta_{\pi}^*i^*}&\dOnto^{i^*}&&\\
C^{(0,p^*)} (G,\overline{\pi})&\lTo_{\ \ \ \ \ \ \ \ \ \delta_{\pi}^*\ \ \ \ \ \ \ }&(\ker d_{\pi})^* & &\\ 
C^{(0,p)} (G,\pi)&\rTo_{\ \ \ \ \ \ \ \ \ \delta_{\pi} \ \ \ \ \ \ \ }& \ker d_{\pi}&&\\
&\luTo(2,4)^{\, \delta_{\overline{\pi}}^*}&\dInto_{i}&&\\
&&C^{(1,p)}_- (G,\pi) & \pile{\rTo^{\ \ \ \ \ \ \ d_{\pi}\ \ \ \ \ \ \ }\\ \rTo_{D}}&C^{(2,p)} (G,\pi)\\
&& \dOnto_{ \overline{i}^{\,*}}& & \\
&&(\ker d_{\overline{\pi}})^* &&\\
\end{diagram}
For the purposes of this proof we will view, abusing the notation, $\delta_{\pi}$ as an operator $C^{(0,p)} (G,\pi)\to \ker d_{\pi}$. 
The natural injection of $\ker d_{\pi}$ into $C^{(1,p)}_- (G,\pi)$ 
will be denoted by $i$. Consequently, the formula (\ref{equation : formula for delta*}) expresses the composition $\delta_{\pi}^*\circ i^*$ in this notation, as on the diagram.
Similar notation is used on the dual level, with $\overline{i}$ denoting the inclusion of $\ker d_{\overline{\pi}}$ into $C^{(1,p^*)}(G,\overline{\pi})$.

By Theorem \ref{theorem : inequality for kappas}, if $2^{-\frac{1}{p}} \kappa_{p}(S,X)<1$ 
we conclude that $\delta_{\overline{\pi}}^*\circ \overline{i}^*$ 
is injective with closed image when restricted to  $\ker d_{\pi}$. In fact, this means that the composition 
$\delta_{\overline{\pi}}^*\circ \overline{i}^*\circ i$ is injective with closed image. 
In particular, $\overline{i}^{\,*}\circ i$ is injective with closed image, and thus its dual, 
$i^*\circ \overline{i}:\ker d_{\overline{\pi}}\to (\ker d_{\pi})^*$, is surjective.

A similar argument applied to $2^{-\frac{1}{p^*}} \kappa_{p^*}(S,X^*)<1$ yields 
$\delta_{\pi}^*\circ i^*\circ \overline{i}$ is also
injective with closed image, which implies that $\delta_{\pi}^*$ is injective with closed image on the image of $i^*\circ \overline{i}$.
Since the latter is surjective, $\delta_{\pi}^*$ is injective with closed image on $(\ker d_{\pi})^*$. 
This on the other hand implies that $\delta_{\pi}$ is onto, which proves the theorem
by Proposition \ref{proposition : delta onto ker d implies fixed point}.
\end{proof}

\begin{rem}\label{remark : getting rid of one spectral condition}
\normalfont
Note that under the assumptions of Theorem \ref{main theorem : intro},
$(\ker d_{\pi})^*$ and $\ker d_{\overline{\pi}}$ are isomorphic (a similar fact holds
for $\ker d_{\pi}$ and $(\ker d_{\overline{\pi}})^*$). These spaces
are closely related to the spaces of cocycles for the given representations.

It is an interesting question, for which isometric representations is $i^*\circ \overline{i}$ automatically an isomorphism, or at least is surjective.
This property would eliminate, for such representations, the need to use the inequality 
$2^{-\frac{1}{p}} \kappa_{p}<1$, which is necessary 
in the above proof.
A special case is discussed and applied in Section \ref{section : p-Laplacian}.
\end{rem}

\begin{rem}\normalfont
Note that it is not clear whether the above method can be extended to subspaces of $Y\subseteq X$. 
This would require estimating $\kappa_p(S,X)$ for some $p$, together with $\kappa_{p^*}(S,X^*)$, 
where $Y^*$ is a quotient of $X^*$.
\end{rem}

\section{$\widetilde{A}_2$ groups}
In this section we apply Theorem \ref{main theorem : intro} to specific groups and Banach spaces.
In \cite{cartwright-et-al} the authors studied a family of groups
$\{G_q\}$ called the $\widetilde{A}_2$-groups. These groups were introduced in 
\cite{cartwright-et-al-GD}, see also \cite{bekka-et-al} for a detailed discussion.
The group $G_q$ has a presentation, 
whose associated link graph $\mathcal{L}(S)$ is the incidence graph of a finite projective 
plane $\mathbf{P}^2(\mathbb{F}_q)$ (here, $q$ is a power of a prime number). Spectra
of such graphs, with weight $\omega\equiv 1$, were computed by Feit and Higman \cite{feit-higman}, see also \cite{bekka-et-al,zuk-gafa}. 
It follows that
$$\kappa_2(S,\RR)=\left(1-\dfrac{\sqrt{q}}{q+1}\right)^{-\frac{1}{2}}.$$
In general, any estimates of $p$-Poincar\'{e} constants are difficult to obtain. In our case,
the link graphs are finite graphs and we can use norm inequalities and a version
of \ref{subsection : behavior under isomorphisms} to give the necessary estimates.
\begin{thm}\label{theorem : fixed points for A_2-groups}
For each $q=k^n$ for some $n\in \NN$ and prime number $k$ we have $$H^1(G_q,\pi)=0$$
for all 
$$2\le p<\dfrac{\ln(q^2+q+1)+\ln(q+1)}{\dfrac{1}{2}\ln(2(q^2+q+1)(q+1))-\ln(2)-\ln\left(\sqrt{1-\dfrac{\sqrt{q}}{q+1}}\right)}$$
and for any isometric representation $\pi$ of $G_q$ on $L_p(Y,\mu)$ fon any
measure space.
\end{thm}

\begin{proof}
We proceed by estimating $\kappa_p(S,L_p)$ and applying Theorem \ref{main theorem : intro}.
Recall that given the space $\ell_p(\Omega)$ for  $2\le p$, where the set $\Omega$ is finite, the following norm
inequalities hold,
$$\Vert f\Vert_{\ell_p(\Omega)}\le \Vert f\Vert_{\ell_2(\Omega)}\le (\#\Omega)^{\frac{1}{2}-\frac{1}{p}}\Vert f \Vert_{\ell_p(\Omega)},$$ 
where $\Vert f\Vert_{\ell_p(\Omega)}=\left(\sum_{x\in \Omega} \vert f(x)\vert^p\right)^{\frac{1}{p}}$. Since the degree of the 
incidence graphs of finite projective planes is
constant and equal to $q+1$ we obtain for $f:S\to \RR$ satisfying $Mf=0$,
\begin{eqnarray*}
\Vert f\Vert_{(1,p)}&=&(q+1)^{\frac{1}{p}} \Vert f\Vert_{\ell_p(S,X)}\\
&\le& (q+1)^{\frac{1}{p}-\frac{1}{2}}  \Vert f\Vert_{(1,2)}\\
&\le& (q+1)^{\frac{1}{p}-\frac{1}{2}} \kappa_2(S,L_2)\Vert \nabla f\Vert_{\ell_2(\mathcal{E},X)}\\
&\le& (q+1)^{\frac{1}{p}-\frac{1}{2}} \kappa_2(S,L_2)\left(\frac{\omega(E)}{2}\right)^{\frac{1}{2}-\frac{1}{p}}\Vert \nabla f\Vert_{\ell_p(\mathcal{E},X)}.
\end{eqnarray*}
For each $q$ we have $\omega(E)=2(q^2+q+1)(q+1),$
which gives the inequality
\begin{eqnarray*}
2^{-\frac{1}{p}} \kappa_2(S,L_p)&\le& 2^{-\frac{1}{p}}  (q+1)^{\frac{1}{p}-\frac{1}{2}} \kappa_2(S,L_2)\left(\frac{\omega(E)}{2}\right)^{\frac{1}{2}-\frac{1}{p}}\\
&=&2^{-\frac{1}{p}}  \left(\sqrt{1-\dfrac{\sqrt{q}}{q+1}}\right)^{-1}   \left(q^2+q+1\right)^{\frac{1}{2}-\frac{1}{p}}.
\end{eqnarray*}
Bounding the above quantity by 1 from above gives 
$$p<\dfrac{2\ln(2(q^2+q+1))}{\ln(2(q^2+q+1))-\ln \left(2\left(1-\dfrac{\sqrt{q}}{q+1}\right)\right)}.$$
A similar norm estimate for $p^*\le 2$, by virtue of the inequality
$$\Vert f\Vert_{\ell_2(\Omega)}\le \Vert f\Vert_{\ell_{p^*}(\Omega)}\le (\#\Omega)^{\frac{1}{p}-\frac{1}{2}}
\Vert f\Vert_{\ell_2(\Omega)},$$  yields
$$p^*> \dfrac{2\ln((q+1)(q^2+q+1))}{\ln\left(2(q^2+q+1)(q+1)\right)+\ln \left(2\left(1-\frac{\sqrt{q}}{q+1}\right)\right)}$$
Simplifying and  comparing $p$ and $\dfrac{p^*}{p^*-1}$ we obtain the claim.
\end{proof}

\begin{rem}\normalfont
The same argument gives a similar conclusion 
for the Banach space $X=\left(\bigoplus X_i\right)_p$, the $\ell_p$-sum in which
$X_i$ is finite-dimensional with a norm sufficiently, and uniformly in $i$, close  to the Euclidean norm. 
We leave the details to the reader.
\end{rem}

\begin{rem}\normalfont
The largest value of $p$ in Theorem \ref{theorem : fixed points for A_2-groups} is approximately 2.106 attained for $q=13$.
As $q$ increases to infinity the values of $p$, for which cohomology vanishes, 
converge to 2 from above.
\end{rem}

\begin{rem}\normalfont
Although our estimate of the constant in the $p$-Poincar\'{e} inequality is not
expected to be optimal, other interpolation methods do not seem to yield significantly better constants. 
For instance, Matousek's interpolation
method for $p$-Poincar\'{e} inequalities \cite{matousek} 
gives a constant strictly greater than $2^{1/p}$ for any $p\ge 2$, since
it  emphasizes independence of dimension and is much better suited to deal with sequences of graphs 
(e.g. expanders).
\end{rem}
Recall that the Banach-Mazur distance $d_{BM}(x,y)$ between two Banach spaces is the infimum of the set of numbers $L$, for which there exists an isomorphism $T:X\to Y$ 
satifying $\Vert x\Vert_X\le \Vert T_x\Vert_Y\le L \Vert x\Vert_X$.
Another consequence of Theorem \ref{main theorem : intro} is that we obtain vanishing of 
cohomology for representations on Banach spaces, whose Banach-Mazur distance from the Hilbert space is controlled. We phrase this property in terms of uniformly bounded representations.
\begin{thm}\label{theorem: A2tilde vanishing and Banach Mazur distance}
Let $G_q$ be an $\widetilde{A}_2$-group and  $\pi$ be a uniformly bounded 
representation of $G_q$ on a Hilbert space $H$, satsfying
$$\sup_{g\in G}\Vert \pi_g\Vert < \sqrt{2\left(1-\dfrac{\sqrt{q}}{q+1}\right)}.$$ 
Then 
$H^1(G,\pi)=0$. 
\end{thm}
\begin{proof}
Let $\Vert v\Vert'=\sup_{g\in G}\Vert \pi_gv\Vert$. Then $\Vert \cdot\Vert'$ is a norm
and $\pi$ is an isometric representation 
on $X=(H,\Vert\cdot\Vert')$. 
The identity is an isomorphism $\mathcal{I}:X\to H$ with
$L= \sqrt{2\left(1-\dfrac{\sqrt{q}}{q+1}\right)}$, and $L\mathcal{I}:X^*\to H$ 
is an isomorphism with the same property. 
The estimate now follows by letting $p=2$ and using the relation between $\kappa_2(S,X)$, $\kappa_2(S,X^*)$
and $\kappa_2(S,H)$, described in
\ref{subsection : behavior under isomorphisms}.
\end{proof}
A similar
fact (with appropriate constants) holds for $L_p$ spaces, for the range of $p$ 
as in the previous theorem.

\section{Hyperbolic groups}\label{section : hyperbolic groups}
In this section we discuss the consequences of \ref{main theorem : intro} in the case of random hyperbolic groups.
In \cite{zuk-gafa} \.{Z}uk used spectral methods to show that many random groups
have property (T) with overwhelming probability. A detailed account was recently
provided in \cite{kotowscy}. We sketch the strategy of the proof and generalize it to $L_p$-spaces. 

In \cite{friedman} it was shown that for a certain random graphs on $n$ vertices of degree $\deg$ 
there exists a constant such that for any $\varepsilon>1$ we have
\begin{equation}\label{equation : lambda_1 for random graphs}
\lim_{n\to \infty} \mathbb{P} \left( \kappa_2(S,\RR)\le 
\left( 1-\left(\dfrac{\sqrt{2\deg}(\deg-1)^{\frac{1}{4}}}{\deg}+\dfrac{\varepsilon}{\deg} \right)  \right)\right)\to 1.
\end{equation}
In \cite{zuk-gafa} a modified link graph, denoted $L'(S)$, with multiple edges was considered. 
$L'(S)$ decomposes into random 
graphs as above and it is shown, using the above estimate, that it has a spectral gap strictly large than $1/2$ with probability 1. 
In our setting, the modified link graph $L'(S)$ can be viewed as a link graph with an admissible weight
$\omega(s,t)$, which is defined to be the number of edges connecting $s$ and $t$. 
Thus we can apply Theorem \ref{main theorem : intro}.
Recall that in the Gromov model $\mathcal{G}(n,l,d)$ for random groups one chooses a 
density $0<d<1$ and considers a group given by a generating set $S$ of cardinality $n$ and $(2n-1)^{ld}$
relations  of length $l$, chosen at random, letting $l$ increase to infinity.

\begin{thm}[\cite{zuk-gafa}; see also \cite{kotowscy} for a detailed proof]
Let $G$ be a random group in the density model, where $1/3< d< 1/2$.
Then, with probability 1, $G$ is hyperbolic and there exists a group $\Gamma$ 
and a homomorphism $\phi:\Gamma\to G$ with the following properties:
\begin{enumerate}
\item $\Gamma$ has a generating set $S$, whose link graph satisfies $2^{-1/2}\kappa_2(S,L_2)<1$,
\item $\phi(\Gamma)$ is of finite index in $G$.
\end{enumerate}
\end{thm}
Given the above, we apply similar norm inequalities as in the case of $\widetilde{A}_2$-groups 
to the link graph of $\Gamma$
and, as before, obtain fixed point properties for affine isometric actions of the group $\Gamma$ 
on $L_p$ for certain $p>2$.  
For any given $p>1$, the property of having vanishing cohomology $H^1(G,\pi)$ for \emph{all} isometric
representations $\pi$ on $L_p$-spaces passes to quotients and from finite index subgroups
to the ambient group. We thus have
\begin{thm}\label{thm: formula for p for hyperbolic}
With the assumptions of the previous theorem, with probability 1, Theorem \ref{main theorem : intro}
applies to hyperbolic groups. More precisely, let  $G$, $\Gamma$ and $\phi$ 
be as above, and let 
$\mathcal{L}(S)=(\mathcal{V},\mathcal{E})$ denote the link graph of $\Gamma$.
Then  $H^1(G,\pi)=0$ for every isometric representation $\pi$  of $G$ on $L_p$ for 
$$p< \min\left\{p_0,{\overline{p}_0}^*\right\},$$
where
$$p_0=\dfrac{\ln\deg_{\omega}-\ln(2\#\mathcal{E})}{\dfrac{1}{2}\ln\left(\dfrac{\deg_{\omega}}{\#\mathcal{E}}\right) -\ln\kappa_2(S,\mathbb{R})}\ \ \ \text{ and }\ \ \ \  \overline{p}_0=\dfrac{\ln(\#\mathcal{V}\deg_{\omega})-\ln 2}{\dfrac{1}{2}\ln(\#\mathcal{V}\deg_{\omega})-\ln\kappa_2(S,\mathbb{R})}.$$
\end{thm}

We also have an estimate the norms of uniformly bounded representations to which cohomological vanishing can be extended for random hyperbolic groups.
\begin{thm}\label{theorem: vanishing hyperbolic and Banach Mazur distanc}
Let $G$ be a hyperbolic group in the Gromov model as above with $\dfrac{1}{3}<d<\dfrac{1}{2}$ and  $\pi$ be 
a uniformly bounded 
representation of $G$ on a Hilbert space $H$, satisfying
$$\sup_{g\in G}\Vert \pi_g\Vert< \dfrac{\sqrt{2}}{\kappa_2(S,\RR)},$$
Then 
$H^1(G,\pi)=0$. 

In other words, the $H^1(G,\pi)$ vanishes with probability 1 for representations bounded by $\sqrt{2}$.
\end{thm}
We remark, that in (\ref{equation : lambda_1 for random graphs}), $\kappa_2(S,\RR)$ tends to 1 as $\deg\to \infty$. Thus the above
upper bound on the norm of the representation is $\sqrt{2}$ with probability 1. 
On the other hand, Shalom showed that the $\operatorname{Sp}(n,1)$ has
non-trivial cohomology with respect to some uniformly bounded representations 
(unpublished). 
The same property for hyperbolic groups is conjectured by Shalom and our Theorem \ref{theorem: vanishing hyperbolic and Banach Mazur distanc}
shows that one can extend property (T) to  class of representations.

We also note that M.~Cowling proposed to define a numerical invariant of a hyperbolic group by setting
$\inf \{ \sup_{g\in G}\Vert \pi_g\Vert \colon H^1(G,\pi)\neq 0\}$.
Theorem \ref{theorem: vanishing hyperbolic and Banach Mazur distanc} gives
a uniform lower bound of $\sqrt{2}$ on such an invariant, 
with probability 1, for hyperbolic groups in the Gromov model.

Theorem \ref{thm: formula for p for hyperbolic} brings us to  another interesting connection.
Pansu \cite{pansu-conformal} defined a quasi-isometry invariant  of a hyperbolic group,
called the conformal dimension, to be the number
$$\operatorname{confdim}(\partial G)=\inf\setd{\dim_{\operatorname{Haus}}(\partial G,d)}{d \text{ quasi-conformally equivalent to  }d_v},$$
where $\dim_{\operatorname{Haus}}$ denotes the Hausdorff dimension, $\partial G$ denotes the Gromov boundary of the hyperbolic group $G$
and  $d_v$ denotes any visual metric on $\partial G$.
We refer to \cite{kapovich-benakli} for a brief overview of conformal dimension of boundaries of
hyperbolic groups.
Bourdon and Pajot \cite{bourdon-pajot} showed that a hyperbolic group acts by affine isometries without
fixed points on $L_p(G)$ for $p$ greater than the conformal dimension of $\partial G$.
Combining this with vanishing of cohomology as studied here we see, that
if $H^1(G,\pi)$ vanishes for all isometric representations on $L_p$ then 
$$p\le \operatorname{confdim}(\partial G).$$  
Gromov \cite[9.B (g)]{gromov} and Pansu \cite[IV.b]{pansu-conformal} posed the question of estimating the conformal dimension
of random hyperbolic groups. Using Theorem \ref{main theorem : intro} we obtain such estimates.
\begin{cor}
With the assumptions and notation of Theorem \ref{thm: formula for p for hyperbolic}, 
$$\operatorname{confdim}(\partial G)\ge   \min\left\{p_0,{\overline{p}_0}^*\right\}.$$
\end{cor}

Finally, as mentioned in the introduction, the above facts show that the
method of Poincar\'{e} inequalities cannot in general give vanishing of cohomology as 
studied in this paper, for all $2<p<\infty$. 
In addition, we have the following quantitative statement about Poincar\'{e} constants.
\begin{cor}
For any hyperbolic group $G$ and any generating set $S$ not containing the identity element, 
the Poincar\'{e} constants on the link graph associated to $S$ satisfy $$\kappa_p(S,L_p)\ge 2^{\frac{1}{p}}\ \ \ \ \text{ or } \ \ \ \ \kappa_{p^*}(S,L_{p^*})\ge 2^{\frac{1}{p^*}},$$
for $p>\operatorname{confdim}(\partial G)$.
\end{cor}

\section{Other Applications}
\subsection{Actions on the circle}
Fixed point properties for the spaces $L_p$, $p>2$, can be applied to studying actions on
the circle, by applying  the vanishing of cohomology to the $L_p$-Liouville  cocycle.
In \cite{navas} the following theorem was proved.
\begin{thm}
Let $G$ be a discrete group, such that $H^1(G,\pi)=0$ for every isometric representation of $G$
on $L_p$ for some $p>2$. Then for every $\alpha>\dfrac{1}{p}$ every homomorphism 
$h:G\to \operatorname{Diff}^{1+\alpha}_+(S^1)$ has finite image.
\end{thm}

Combining this result with, for instance, Theorem \ref{theorem : fixed points for A_2-groups} we 
obtain
\begin{cor}
Let $q$ be a power of a prime number and 
$G_q$ be be the corresponding  $\widetilde{A}_2$ group. Then every homomorphism 
$h:G\to \operatorname{Diff}^{1+\alpha}_+(S^1)$ has finite image for 
$$\alpha >  
\dfrac{\dfrac{1}{2}\ln(2(q^2+q+1)(q+1))-\ln(2)-
\ln\left(\sqrt{1-\dfrac{\sqrt{q}}{q+1}}\right)}{\ln(q^2+q+1)+\ln(q+1)}.$$
\end{cor}

\subsection{The finite--dimensional case and the $p$-Laplacian}\label{section : p-Laplacian}

Let $1<p<\infty$. The $p$-Laplacian $\Delta_p$ is an operator $\Delta_p:\ell_p(V)\to\ell_p(V)$, 
defined by the formula
$$\Delta_p f(x)=\sum_{x\sim y} (f(x)-f(y))^{[p]}\omega(x,y),$$
for $f:V\to \RR$,
where $a^{[p]}= \vert a\vert^{p-1} \sign(a)$.
The $p$-Laplacian reduces to the standard discrete Laplacian for $p=2$, and
is non-linear when $p\neq 2$.
The $p$-Laplacian is of great importance in the study of partial differential equations.
Its discrete version was studied e.g. in \cite{amghibech,takeuchi}

A real number $\lambda$ is an eigenvalue of the $p$-Laplacian $\Delta_p$
if there exists a function $f:V\to \RR$ satisfying
$$\Delta_pf=\lambda f^{[p]}.$$
The eigenvalues of the $p$-Laplacian are difficult to 
compute in the case 
$p\neq2$, due to non-linearity of $\Delta_p$, see \cite{grigorchuk-nowak} for 
explicit estimates.
Define
\begin{equation}\label{equation : lambda_1}
\lambda_1^{(p)}(\Gamma)=\inf\left\{\dfrac{\sum_{x\in V}\sum_{y\sim x}\vert f(x)-f(y)\vert^p \omega(x,y) }
{\inf_{\alpha\in \RR} \sum_{x\in V}\vert f(x)-\alpha\vert^p\deg_{\omega}(x)}\right\},
\end{equation}
with the infimum taken over all $f:V\to \RR$ such that $f$ is not constant.
$\lambda_1^{(p)}$ is  the smallest positive eigenvalue of the discrete
$p$-Laplacian $\Delta_p$ or the \emph{$p$-spectral gap}.

We now apply an estimate similar to the one in Corollary 
\ref{corollary : Poincare constant on Cayley graph}, to finite quotients of groups.
Let $G$ be a finitely generated group and 
consider a homomorphism $h:G\to H$, where $H$ is a finite group. Let $p>1$ and let
$\ell_p^{\,0}(H)$ denote the subspace of $\ell_p(H)$ consisting of those functions,
which sum to 0. 

We can identify the dual $\ell_p^{\,0}(H)^*$ with
the space $\ell_{p^*}^{\,0}(H)$, with the norm
$$\Vert f \Vert =\inf_{\alpha\in \RR}\Vert f-\alpha\Vert_{p^*}.$$
We will use our results to estimate the $p^*$-spectral gap for this norm on the Cayley graph of $H$.

Let $X=\ell_{p}^{\,0}(H)^*$ be equipped with the adjoint of the left regular representation 
$\lambda$ on $\ell_p(H)$, restricted to $X^*=\ell_p^{\,0}(H)$. We have
$$\kappa_p(S, X^*)\le \kappa_p(S,\ell_p(H))=\kappa_p(S,\RR).$$
Computing the Poincar\'{e} constant of the link graph for the norm of $X$ is not
straightforward. 
However, following the strategy outlined in Remark 
\ref{remark : getting rid of one spectral condition},
we will show that we can bypass this condition. In order to do this we need to show that
$i^*\circ\overline{i}$ is onto. In fact, a stronger statement is true.

\begin{lem}\label{lemma: two kernels of d isomorphic}
Under the above assumptions, the map $i^*\circ \overline{i}:\ker d\to (\ker \overline{d})$, is an isomorphism. 
\end{lem}

\begin{proof}
We can view $X$ and $X^*$ as having the same underlying vector space (real-valued
functions $f:H\to X$ with mean value 0), equipped with two different norms.
Similarly, $C^{(1,p^*)}_-(G,\pi)$ and $C^{(1,p)}_-(G,\overline{\pi})$ also have the same underlying 
vector space, equipped with two different norms. The adjoint
$\overline{\lambda}$ of the left regular representation, coincides with $\lambda$. 
For that reason $\ker d_{\pi}$ and $\ker d_{\overline{\pi}}$ describe the same vector subspace.
The claim follows from the fact that all the spaces involved are finite-dimensional and
complemented.\end{proof}

Now, since the representation of $G$ on $X$ does not have invariant vectors
and $\delta_{\pi}$ is onto $\ker d_{\pi}$, we can conclude, by the Open Mapping Theorem, that 
$\delta_{\pi}$ in fact induces an isomorphism between $C^{(0,p^*)}(G,\pi)$ and $\ker d_{\pi}$.
It follows from Theorem \ref{theorem : inequality for kappas}, that 
$$ 2 \left(1-2^{-\frac{1}{p}}\kappa_p(S,\RR)\right) \Vert f\Vert_{(0,p^*)}\le \Vert \delta_{\overline{\pi}} f \Vert_{1,p^*}.$$
Since $f\in \ell_p^{\,0}(H)$, this gives
$$
 \left(2 \left(1-2^{-\frac{1}{p}}\kappa_p(S,\RR)\right)\right)^{p^*} \Vert f\Vert_{X}^{p^*}\le 
 \sum_{s\in S}\Vert f-\lambda_s f\Vert_X^{p^*}\dfrac{\deg_{\omega}(s)}{\omega(E)}.
 $$
Since $\Vert v\Vert_X\le \Vert v\Vert_{\ell_p(H)}$, this yields
\begin{equation*}
\begin{split} 
\left(2 \left(1-2^{-\frac{1}{p}}\kappa_p(S,\RR)\right)\right)^{p^*}
& \inf_{c\in \RR}\sum_{h\in H} \vert f(h)-c\vert^{p^*}\deg_{\omega}(h)\\
 &\le \sum_{h\in H}\sum_{g\sim h}\vert  f(h)-f(g)\vert^{p^*} \dfrac{\deg_{\omega}(g^{-1}h)}{\omega(E)}.
 \end{split}
 \end{equation*}
(Note that $\deg_{\omega}(g^{-1}h)$ refer to $\mathcal{L}(S)$, not the Cayley graph of $H$.)
\begin{cor}
Let 
$G$ be a group generated by a finite, symmetric set $S$ not containing the
identity element. If the link graph $\mathcal{L}(S)$ is connected and for some $1<p<\infty$
the Poincar\'{e} constant satisfies
$$2^{-\frac{1}{p}} \kappa_p(S,\RR)<1,$$
then 
$$\lambda_1^{(p)}\ge  2 \left(1-2^{-\frac{1}{p}}\kappa_p(S,\RR)\right)$$
on the Cayley graph of any finite quotient of $G$, for any weight $\omega(g,h)\ge  \dfrac{\deg_{\omega}(g^{-1}h)}{\omega(E)}$.

\end{cor} 

\begin{rem}\normalfont
A similar claim as in lemma \ref{lemma: two kernels of d isomorphic} holds for any orthogonal
representation which is also isometric on $\ell_p(H)$.
\end{rem}

\bigskip
\footnotesize
\noindent\textit{Acknowledgments.}

I am grateful to Florent Baudier, Pierre-Emmanuel Caprace, Valerio Capraro,
Tadeusz Januszkiewicz and Juhani Koivisto for helpful comments and to Masato Mimura and Andrzej \.{Z}uk
for enjoyable conversations on related topics. I am also greatly indebted to Antoine Gournay  and Alain Valette, who suggested
several improvements and simplifications, in particular the use of the idempotent $Q_{\pi}$.

The author was partially supported by NSF grant DMS-0900874 and by the Foundation for Polish Science.


\begin{thebibliography}{SK}




\normalsize

\bibitem{amghibech}
 Amghibech, S.: 
 Eigenvalues of the discrete p-Laplacian for graphs. Ars Combin. \textbf{67} (2003), 283--302


\bibitem{bader-et-al}
 Bader, U., Furman, A., Gelander, T., Monod, N.:
 Property (T) and rigidity for actions on Banach spaces. 
 Acta Math. \textbf{198} (2007), no. 1, 57--105.
 
\bibitem{ballman-swiatkowski}
 Ballmann, W., \'{S}wi\c{a}tkowski, J.: 
 On $L^2$-cohomology and property (T) for automorphism groups of polyhedral cell complexes. 
 Geom. Funct. Anal. \textbf{7} (1997), no. 4, 615--645.
 
\bibitem{bekka-et-al}
  Bekka, B., de la Harpe, P., Valette, A.: 
  Kazhdan's property (T). 
  New Mathematical Monographs, 11. Cambridge University Press, Cambridge, 2008.
  
  
  
\bibitem{bourdon-pajot} 
 Bourdon, M., Pajot, H.: 
 Cohomologie $l_p$ et espaces de Besov. 
 J. Reine Angew. Math. \textbf{558} (2003), 85--108.
 
 
 
\bibitem{cartwright-et-al-GD}
 Cartwright, D. I., Mantero, A. M., Steger, T., Zappa, A.: 
 Groups acting simply transitively on the vertices of a building of type $\tilde{A}_2$. I. Geom. Dedicata \textbf{47} (1993), no. 2, 143--166. 
 
 
 
 \bibitem{cartwright-et-al}
 Cartwright, D. I., M\l otkowski, W., Steger, T.:
 Property (T) and $\tilde{A}_2$ groups. 
 Ann. Inst. Fourier (Grenoble) \textbf{44} (1994), no. 1, 213--248.
 
 
 
 \bibitem{chatterji-drutu-haglund}
 Chatterji, I., Dru\c{t}u, C., Haglund, F.: 
 Kazhdan and Haagerup properties from the median viewpoint. 
 Adv. Math. \textbf{225} (2010), no. 2, 882--921.
 
 
\bibitem{dymara-januszkiewicz}
Dymara, J., Januszkiewicz, T.: 
Cohomology of buildings and their automorphism groups. 
Invent. Math. \textbf{150} (2002), no. 3, 579--627.

 
\bibitem{ershov-et-al}
Ershov, M., Jaikin-Zapirain, A., 
Property (T) for noncommutative universal lattices. 
Invent. Math. \textbf{179} (2010), no. 2, 303--347. 
 
 \bibitem{feit-higman}
 Feit, W., Higman, G.: 
 The nonexistence of certain generalized polygons. 
 J. Algebra 1 1964 114--131. 
 
 
 \bibitem{fisher-margulis}
 Fisher, D., Margulis, G.:
 Almost isometric actions, property (T), and local rigidity. 
 Invent. Math. \textbf{162} (2005), no. 1, 19Ð80. 
 
 
 \bibitem{friedman}
 Friedman, J.: 
 Relative expanders or weakly relatively Ramanujan graphs. 
 Duke Math. J. \textbf{118} (2003), no. 1, 19--35.
 
 \bibitem{garland}
 Garland, H.: 
 $p$-adic curvature and the cohomology of discrete subgroups of $p$-adic groups. 
 Ann. of Math. (2) \textbf{97} (1973), 375--423.
 
 
\bibitem{grigorchuk-nowak}
 Grigorchuk, R. I., Nowak, P. W.: 
 Diameters, distortion, and eigenvalues. 
 European J. Combin. \textbf{33} (2012), no. 7, 1574--1587. 
 
 
 \bibitem{gromov}
 Gromov, M.: 
 Asymptotic invariants of infinite groups. 
 Geometric group theory, Vol. 2 (Sussex, 1991), 1Ð295, London Math. Soc. Lecture Note Ser., 182, Cambridge Univ. Press, Cambridge, 1993.
 
 \bibitem{izeki-nayatani}
 Izeki, H., Nayatani, S.:
 Combinatorial harmonic maps and discrete-group actions on Hadamard spaces. 
 Geom. Dedicata \textbf{114} (2005), 147--188. 
 
 \bibitem{kapovich-benakli}
 Kapovich, I., Benakli, N.: 
 Boundaries of hyperbolic groups. 
 Combinatorial and geometric group theory (New York, 2000/Hoboken, NJ, 2001), 39Ð93, Contemp. Math., 296, Amer. Math. Soc., Providence, RI, 2002.
 
 \bibitem{kassabov}
 Kassabov, M.:
 Subspace arrangements and property T. 
 Groups Geom. Dyn. \textbf{5} (2011), no. 2, 445--477. 
 
 \bibitem{kotowscy}
 Kotowski, M., Kotowski, M.: 
 Random groups and property (T): \.{Z}uk's theorem revisited. 
 J. Lond. Math. Soc. (2) \textbf{88} (2013), no. 2, 396--416.
 
 
 \bibitem{lafforgue}
 Lafforgue, V.:
 Un renforcement de la propri\'{e}t\'{e} (T).  
 Duke Math. J. \textbf{143} (2008), no. 3, 559--602. 
 
 \bibitem{matousek}
Matou\v{s}ek, J.:
On embedding expanders into $l_p$ spaces. 
Israel J. Math. \textbf{102} (1997), 189--197. 


\bibitem{mimura}
Mimura, M.: 
Fixed point properties and second bounded cohomology of universal lattices on Banach spaces. 
J. Reine Angew. Math. \textbf{653} (2011), 115--134.
 
 
 \bibitem{naor-silberman}
 Naor, A., Silberman, L.:
 Poincar\'{e} inequalities, embeddings, and wild groups. 
 Compos. Math. \textbf{147} (2011), no. 5, 1546--1572.
 
  
 \bibitem{navas-fr}
 Navas, A.:
 Actions de groupes de Kazhdan sur le cercle.  
 Ann. Sci. \'{E}cole Norm. Sup. (4) \textbf{35} (2002), no. 5, 749--758.
 
 
\bibitem{navas} 
Navas, A.: 
Reduction of cocycles and groups of diffeomorphisms of the circle. 
Bull. Belg. Math. Soc. Simon Stevin \textbf{13} (2006), no. 2, 193Ð205. 
 
 \bibitem{nica}
 Nica, B.:
 Proper isometric actions of hyperbolic groups on $L^p$-spaces. 
 Compos. Math. 149 (2013), no. 5, 773--792.
 
 \bibitem{nowak}
 Nowak, P. W.:
 Group actions on Banach spaces and a geometric characterization of a-T-menability. 
 Topology Appl. \textbf{153} (2006), no. 18, 3409--3412.
 
 \bibitem{nowak-corrected}
 Nowak, P. W.:
 Group actions on Banach spaces and a geometric characterization of a-T-menability. 
(corrected version). arXiv:math/0404402v3 

\bibitem{nowak-survey}
Nowak, P. W.:
Group Actions on Banach Spaces. 
Handbook of Group Actions, ed.: by L. Ji, A. Papadopoulos, S.-T. Yau, to appear.


\bibitem{pansu-conformal}
Pansu, P.:
Dimension conforme et sph\`ere \`a l'infini des vari\'et\'es \`a courbure n\'egative.
Ann. Acad. Sci. Fenn. Ser. A I Math. \textbf{14} (1989), no. 2, 177--212.
 
 
 \bibitem{pansu-95}
 Pansu, P.:
 Cohomologie $L_p$: invariance sous quasiisom\'{e}trie.
 Preprint, 1995.
 
 \bibitem{pansu}
Pansu, P.:
Formules de Matsushima, de Garland et propri\'et\'e (T) pour des groupes agissant sur des espaces sym\'etriques ou des immeubles 
Bull. Soc. Math. France 126 (1998), no. 1, 107--139.
 
 \bibitem{takeuchi}
 Takeuchi, H.:
 The spectrum of the $p$-Laplacian and $p$-harmonic morphisms on graphs. 
 Illinois J. Math. \textbf{47} (2003), no. 3, 939--955. 
 
 \bibitem{wang-jdg}
 Wang, M.-T.:
 A fixed point theorem of discrete group actions on Riemannian manifolds. 
 J. Differential Geom. \textbf{50} (1998), no. 2, 249--267.
 
 \bibitem{wang-cag}
 Wang, M.-T.: 
 Generalized harmonic maps and representations of discrete groups. 
 Comm. Anal. Geom. \textbf{8} (2000), no. 3, 545--563.
 
 \bibitem{yu-hyperbolic}
 Yu, G.: 
 Hyperbolic groups admit proper affine isometric actions on $l^p$-spaces. 
 Geom. Funct. Anal. \textbf{15} (2005), no. 5, 1144--1151.

\bibitem{zuk-comptes} 
\.{Z}uk, A.:
La propri\'et\'e (T) de Kazhdan pour les groupes agissant sur les poly\`edres.
 C. R. Acad. Sci. Paris SŽr. I Math. \textbf{323} (1996), no. 5, 453--458.
 
 \bibitem{zuk-gafa}
 \.{Z}uk, A.: 
 Property (T) and Kazhdan constants for discrete groups. 
 Geom. Funct. Anal. \textbf{13} (2003), no. 3, 643--670.
 
 



\end{thebibliography}
\end{document}